\documentclass[11pt]{amsart}
\usepackage[vmargin=2.6cm,hmargin=2.2cm]{geometry} 
\usepackage[colorlinks,citecolor=blue,urlcolor=blue,bookmarks=false,hypertexnames=true]{hyperref}

\usepackage{color}
\usepackage{amsthm}
\usepackage{graphicx}
\usepackage{times}
\usepackage{xcolor}
\graphicspath{ {./images/} }
\usepackage{tikz}
\usepackage{multirow}
\usepackage{caption}
\usepackage{subcaption,mathrsfs,dsfont}
\usepackage[square,numbers]{natbib}
\usepackage{comment}
\usepackage[title]{appendix}
\usepackage{fancyhdr}	
\usepackage{enumerate}
\usepackage{setspace, xspace}
\allowdisplaybreaks[3] 

\numberwithin{equation}{section}

\newtheorem{lemma}{Lemma}[section]

\newtheorem{eg}{Example}[section]
\newtheorem{defn}{Definition}[section]
\newtheorem{assmp}{Assumption}[section]
\newtheorem{thm}{Theorem} [section]
\newcommand{\sgn}[1]{\ensuremath{\operatorname{sgn}\!\left(#1\right)}}
\newcommand{\dd}{\operatorname{d}\! }
\newcommand{\dt}{\operatorname{d}\! t}

\newcommand{\ds}{\operatorname{d}\! s}

\newcommand{\dw}{\operatorname{d}\! W}

\newcommand{\ep}{\varepsilon}

\newcommand{\nn}{\nonumber}

\newcommand{\E}{{\mathbb E}}
\newcommand{\R}{{\mathbb R}}

\newcommand{\BMO}{L^{2, \;\mathrm{BMO}}_{\mathbb F}(0, T;\R^{n})}

\newcommand{\idd}[1]{\ensuremath{\operatorname{\mathds{1}}_{#1}}}
\newcommand{\og}{\mathcal{G}}
\newcommand{\oh}{\mathcal{H}}

\newcommand{\lm}{\Lambda}
\newcommand{\hf}{\frac{1}{2}}
\newcommand{\ito}{It\^{o}'s lemma\xspace} 

\title[Optimal control of stochastic homogenous systems]
{Optimal control of stochastic homogenous systems}
\author[Y.~Hu]{Ying Hu}
\author[X.~Shi]{Xiaomin Shi}
\author[Z.~Q.~Xu]{Zuo Quan Xu}
\date{\today}

\keywords{}
\address{Y.~Hu: Univ Rennes, CNRS, IRMAR-UMR 6625, F-35000 Rennes, France.}
\email{{ying.hu@univ-rennes1.fr}}
\address{X.~Shi:  School of Statistics and Mathematics, Shandong University of Finance and Economics, Jinan
250014, China. }
\email{{shixm@mail.sdu.edu.cn}}
\address{Q. Z.~Xu: Department of Applied Mathematics, The Hong Kong Polytechnic University, Kowloon, Hong Kong, China.}
\email{{maxu@polyu.edu.hk}}

\begin{document}
\setcitestyle{numbers}
\maketitle


\begin{abstract}
This paper investigates a new class of homogeneous stochastic control problems with cone control constraints, extending the classical homogeneous stochastic linear-quadratic (LQ) framework to encompass nonlinear system dynamics and non-quadratic cost functionals. We demonstrate that, analogous to the LQ case, the optimal controls and value functions for these generalized problems are intimately connected to a novel class of highly nonlinear backward stochastic differential equations (BSDEs). We establish the existence and uniqueness of solutions to these BSDEs under three distinct sets of conditions, employing techniques such as truncation functions and logarithmic transformations. Furthermore, we derive explicit feedback representations for the optimal controls and value functions in terms of the solutions to these BSDEs, supported by rigorous verification arguments. Our general solvability conditions allow us to recover many known results for homogeneous LQ problems, including both standard and singular cases, as special instances of our framework.
\end{abstract}

\subsection*{Keywords:} Homogenous system; stochastic LQ problem; cone constraint; BSDE; existence and uniqueness; verification theorem.
\subsection*{Mathematics Subject Classification (2020):} 93E20. 60H10.

\section{Introduction}
Consider a controlled stochastic differential equation (SDE):
\begin{align}\label{state1}
\dd X_t=b_t( X_t, u_t)\dt+\sigma_t( X_t, u_t)^{\top}\dw_t,\quad X_0=x,
\end{align}
and a related cost functional:
\begin{align}\label{costfunction}
J(X, u):=\E\Big[\int_0^T f_t( X_t, u_t)\dt+g(X_{T})\Big],
\end{align}
where $u$ is the control variable, $X$ is the corresponding state variable, and $b$, $\sigma$, $f$, $g$ are given deterministic or stochastic functions.
The stochastic control problem associated with \eqref{state1} and \eqref{costfunction} is to minimize over control processes the cost functional $J$, and determine the associated value function:
\begin{align}\label{prob1}
V(x):=\inf_{u\in\mathcal{U}} J(X, u).
\end{align}
More precise formulation will be given in the next section.

A vast number of papers deal with the above stochastic control problem \eqref{state1}-\eqref{prob1} subject to various conditions. Please refer to the seminal books Yong and Zhou \cite{YZ} and Pham \cite{Pham} for a systematic account on stochastic optimal control.
There are two well-known approaches to tackle the problem nowadays.
The first one is called the dynamic programming. It leads to the study of the so-called Hamilton-Jacobi-Bellman equations when all the functions $b,\sigma,f,g$ are deterministic. The other approach is called the (stochastic) maximum principle, which provides a necessary condition for optimal control; see, e.g., Peng \cite{Peng}, Hu, Ji and Xue \cite{HJX}. It leads to the study of a system of forward and backward stochastic differential equations (FBSDEs) when some of the maps $b,\sigma,f,g$ are stochastic. Generally, such FBSDEs are fully coupled so that they cannot be further reduced to simple forms.

But when the problem \eqref{state1}-\eqref{prob1} is a linear-quadratic (LQ) problem, where the controlled dynamics is in linear form (i.e. $b$ and $\sigma$ are linear functions of $(x,u)$) and the cost functional $J$ in quadratic form (i.e. $f$ and $g$ are quadratic functions of $(x,u)$),
its FBSDEs system can be decoupled through the famous stochastic Riccati equation (SRE). And it admits elegant optimal state feedback control and the optimal cost value in terms of the solution to the SRE, see, e.g., Wonham \cite{Wonham} for deterministic coefficients case. When the coefficients are stochastic processes, Bismut \cite{Bismut76} firstly finds that the related SRE becomes a highly nonlinear backward stochastic differential equation (BSDE). Unfortunately he could not show the existence of such a solution in general. Since then, numerous progresses have been made in solving SREs. Peng \cite{Peng92} proves the existence of a unique solution to the SRE when $\sigma$ dose not contain the control variable.  Kohlmann and Tang \cite{KT} solves the existence and uniqueness issues for one-dimensional SREs, then Tang \cite{Tang03,Tang15} addresses the matrix-valued case using the stochastic maximum principle and dynamic programming method respectively. For the indefinite case, please refer to Du \cite{Du}, Hu and Zhou \cite{HZ03}, Qian and Zhou \cite{QZ}, Sun, Xiong and Yong \cite{SXY} and Yu \cite{Yu}.

\begin{eg}
Suppose that $\underline A$, $\overline A$, $\underline B$, $\overline B$, $C$, $D$, $Q$, $R$ are predictable processes of suitable dimension, $G$ is some random variable, and $p\geq 1$, let us  consider a stochastic control problem with the  following controlled SDE
\begin{align}
\dd X_t=(\underline A_tX_t^+-\overline A_tX_t^-+\underline B_t^{\top}u^+_t-\overline B_t^{\top}u^-_t)\dt+(C_t X_t+D_tu_t)^{\top}\dw_t,\quad X_0=x\in\R,
\end{align}
and cost functional
\begin{align}
J(X, u):=\E\Big[\int_0^T (|Q_tX_t|^p+|R_tu_t|^p)\dt+|GX_{T}|^p\Big].
\end{align}
Noting $\underline A$, $\overline A$, $\underline B$, $\overline B$, $C$, $D$, $Q$, $R$ are stochastic processes and $b_t(x,u)=\underline A_tx^+-\overline A_tx^-+\underline B_t^{\top}u^+-\overline B_t^{\top}u^-$ is not differentiable w.r.t. $(x,u)$, both the dynamic programming approach and the stochastic maximum principle fail to handle this problem.
To the best of our knowledge, no existing results in the literature can cover this example.
\end{eg}
Inspired by LQ problems and the above toy example, this paper proposes a new class of homogeneous stochastic control problems, which extend the classical homogeneous LQ problems to new ones with possibly nonlinear controlled dynamics and non-quadratic cost functionals, but their optimal solutions can be characterized by pure BSDEs, rather than fully coupled FBSDEs. 

To introduce our new class of problems, we first define homogeneous functions as follows.
\begin{defn}[Homogeneous function]\label{homofunc}
Given a real constant $p\neq 0$ and a closed cone $\mathscr{Y}\subseteq \R^{k} $, a function $\varphi:\mathscr{Y} \to\R^{d}$ is called a homogeneous function of degree $p$ in $y\in\mathscr{Y}$ if
\[\varphi( \lambda y)=\lambda^p \varphi( y)\]
holds for any scalar $\lambda > 0$ and $y\in \mathscr{Y} $.
\end{defn}

\begin{eg}\label{eg1}
Suppose $p>0$.
Then the following functions, with coefficients being possibly random, are homogeneous of degree $p$ in $(x,u)\in\R^{k}\times\R^{m}$:
\[ |Ax+Bu|^{p}, ~~ |Ax+B|u||^{p},~~ |x|^p+|u|^p, ~~ (|x|^2+|u|^2)^{p/2}, \]
\[A |x|^{p/2}|u|^{p/2}+B |x|^{p/3}|u|^{2p/3}, ~~|\underline Ax^++\overline A x^-+\underline B^{\top}u^+-\overline B^{\top}u^-|^p,\]
\[\inf_{(A,B)}|Ax+B^{\top}u|^{p},~~
\sup_{(A,B)}|Ax+B^{\top}u|^{p}.\]
As usual, $x^{+}:=(x_1^+,...,x_k^+)^{\top}, \ x^{-}:=(-x)^{+}$ for $x=(x_1,...,x_k)^{\top}\in\R^k$, where $x_j^+=\max\{x_j,0\}$,  $j=1,...,k$.
We remark that any homogeneous function in $(x,u)$ multiplied by $F\Big(\tfrac{u}{|x|}\idd{x\neq 0}\Big)$ is still a homogeneous function of the same degree, where $F$ is any function.
\end{eg}

In this paper, we only focus on scalar-valued state process, namely $X_{t}\in\R$ in the process \eqref{state1} and the controls  are constrained  to a closed cone $\Gamma\subseteq\R^{m}$. The stochastic control problem \eqref{state1}-\eqref{prob1}  is supposed to satisfy the following homogeneous assumption.
\begin{assmp}[Homogeneous system]\label{homocondition}
For all $t\in[0,T]$, the dynamic coefficients $b_{t}$ and $\sigma_{t}$ are homogeneous functions of degree $1$ in $(x,u)\in\R\times \Gamma$; and the cost functional coefficients $f_{t}$ and $g$ are homogeneous functions of degree $p\neq 0$ in $(x,u)\in\R\times\Gamma$ and $x\in\R$ respectively. 
\end{assmp}

By setting $\lambda=2$ and $y=0$ in Definition \ref{homofunc}, we see $\varphi(0)=0$ if $\varphi$ is homogeneous of degree $p\neq 0$. Therefore, all the coefficients $b_{t}$, $\sigma_{t}$, $f_{t}$ and $g$ are zero at the origin of coordinates under Assumption \ref{homocondition}. We will use this fact without claim in the future.

\begin{defn}[Homogeneous stochastic control problem]
The  problem \eqref{state1}-\eqref{prob1}  is called a homogeneous stochastic control problem of degree $p$, if the coefficients $b,\sigma,f,g$ satisfy Assumption \ref{homocondition} with $p\neq 0$.
\end{defn}

Clearly, the classical homogeneous LQ control problems are special homogeneous problems of degree $2$, therefore the new class of homogeneous problems, which allows non-linear state processes and non-quadratic cost functionals, is an extension of the homogeneous LQ problems.

This paper focuses on the solvability of the homogeneous stochastic control problems under three different conditions.
Although mainly focus on the case $p>1$, many parallel results can be established for the case $p<0$. 
We encourage the interested readers to give the details. The remainder case  is complex.

The idea to solve homogeneous stochastic control problems stems from Hu and Zhou \cite{HZ} and Ji, Jin and Shi \cite{JJS}. It is well known that for LQ stochastic control problems, the optimal controls take state feedback form through the famous SREs, see, e.g., Tang \cite{Tang03}. By performing Tanaka's formula to $X^+$ and $X^-$ respectively, it is proved in \cite{HZ} that the cone-constrained stochastic LQ problem admits a piecewise linear feedback optimal control expressed in terms of the solutions to two BSDEs. The above method (combined with some extra convex duality analysis) still works for mean-variance portfolio selection problems with a specific kind of non-linear wealth dynamic studied
in \cite{JJS}. In the models of \cite{HZ,JJS}, the optimal (translated) state processes will never cross $0$, i.e., it will remain positive (resp. negative) if the initial state is positive (resp. negative). This phenomenon actually results from the homogeneity of the control system. Therefore, our new homogeneous stochastic control problems will inherit this property.

In this paper, we extend the models of \cite{HZ,JJS} to a new class of homogeneous stochastic control problems with possibly non-linear state processes and non-quadratic cost functionals. We will show that the problem \eqref{state1}-\eqref{prob1} admits an optimal control, in a state feedback form via two   highly non-linear  BSDEs.
We firstly prove the existence and uniqueness of solutions to these two BSDEs based on truncation function technique, log transformation and some delicate analysis.
The famous result of Kobylanski \cite{Ko} on quadratic growth BSDEs plays a fundamental tool in our argument. Finally, the solutions of the associated BSDEs are used to construct a candidate control which is verified to be admissible, and optimal to the original homogeneous stochastic control problem \eqref{state1}-\eqref{prob1}.
We remark that the uniqueness of solutions to the BSDEs in \cite{HZ} was obtained as a by product of the verification theorem. Here we will provide a direct proof, based on the bounded mean oscillation (BMO) martingale theory, which is interesting in its own from the point of view of BSDE theory.

The rest part of this paper is organized as follows.
In Section \ref{fm}, we present some notation and the precise assumptions on the homogeneous stochastic control problem \eqref{state1}-\eqref{prob1}. Section \ref{BSDE} is devoted to establishing the solubility of two associated BSDEs under three different conditions.
In Section \ref{Verification}, we determine the optimal controls and optimal values to the homogeneous stochastic control problems via the solutions to the two BSDEs through a straightforward and rigorous verification theorem. Some concluding remarks are given in Section \ref{conclude}.

\section{Preliminaries}\label{fm}
Let $(\Omega, \mathcal F, \mathbb{P})$ be a fixed complete probability space on which are defined a standard $n$-dimensional Brownian motion $\{W_t\}_{t\geq0}$. Let $\{\mathcal{F}_t\}_{t\geq0}$ be the natural filtration of the Brownian motion $W$ augmented by the $\mathbb{P}$-null sets of $\mathcal{F}$. Let $\mathcal{P}$ be the predictable $\sigma$-field of $[0,T]\times\Omega$.
We stipulate that, in what follows, all the processes under consideration, unless otherwise stated, are $\mathcal{P}$ measurable.

We denote by $\R^m$ the set of $m$-dimensional column vectors, by $\R^m_+$ the set of vectors in $\R^m$ whose components are nonnegative, by $\R^{m\times n}$ the set of $m\times n$ real matrices, by $\mathbb{S}^m$ the set of symmetric $m\times m$ real matrices, and by $\mathbf{1}_m$ the $m$-dimensional identity matrix. Therefore, $\R^m\equiv\R^{m\times 1}$. 
For any matrix $M=(m_{ij})\in \R^{m\times n}$, we denote its transpose by $M^{\top}$, and its norm by $|M|=\sqrt{\sum_{ij}m_{ij}^2}$. If $M\in\mathbb{S}^n$ is positive definite (resp. positive semidefinite), we write $M>$ (reps. $\geq$) $0.$ We write $A>$ (resp. $\geq$) $B$ if $A, B\in\mathbb{S}^n$ and $A-B>$ (resp. $\geq$) $0.$

We fix a constant $T>0$ throughout the paper to denote the control horizon. We use the following notation throughout the paper:
\begin{align*}
L^2_{\mathcal{F}_T}(\Omega;\R)&=\Big\{\xi:\Omega\rightarrow
\R\;\Big|\;\xi\mbox { is }\mathcal{F}_{T}\mbox{-measurable, and }\E(|\xi|^{2})%
<\infty\Big\}, \\
L^{\infty}_{\mathcal{F}_T}(\Omega;\R)&=\Big\{\xi:\Omega\rightarrow
\R\;\Big|\;\xi\mbox { is }\mathcal{F}_{T}\mbox{-measurable, and essentially bounded}\Big\}, \\
L^{0}_{\mathbb F}(0, T;\R)&=\Big\{\phi:[0, T]\times\Omega\rightarrow
\R\;\Big|\;\phi\mbox{ is $\mathcal{P}$-measurable}
\Big\}, \\
L^{2,loc}_{\mathbb F}(0, T;\R)&=\Big\{\phi\in L^{0}_{\mathbb F}(0, T;\R)\;\Big|\; \int_{0}^{T}|\phi_t|^{2}\dt<\infty, \ a.s.\Big\}, \\
L^{2}_{\mathbb F}(0, T;\R)&=\Big\{\phi\in L^{0}_{\mathbb F}(0, T;\R)\;\Big|\; \E\int_{0}^{T}|\phi_t|^{2}\dt<\infty\Big\}, \\
L^{\infty}_{\mathbb F}(0, T;\R)&=\Big\{\phi\in L^{0}_{\mathbb F}(0, T;\R)\;\Big|\; \mbox{$\phi$ is essentially bounded}
\Big\},\\
\BMO&=\bigg\{\lm \in L^{2}_{\mathbb F}(0, T;\R^{n}) \;\bigg|\; \int_0^\cdot\lm_s^{\top}\dw_{s} \mbox{ is a BMO martingale on $[0, T]$}\bigg\}.
\end{align*}
These definitions are generalized in the obvious way to the cases that $\R$ is replaced by $\R^n$, $\R^{n\times m}$ or $\mathbb{S}^n$.
For the definition of BMO martingales, we refer the readers to \cite{Ka} or \cite{HSX}. Here we will use the following fact about BMO martingales: the Dol\'eans-Dade stochastic exponential $\mathcal{E}\big(\int_0^\cdot\lm'\dw\big)$ is a uniformly integrable martingale if $\lm\in\BMO$.

In our argument, $s$, $t$, $\omega$, ``$\mathbb{P}$-almost surely'' and ``almost everywhere'', will be suppressed for simplicity in many circumstances, when no confusion occurs.
We shall use $c$, which can be different at each occurrence, to represent some positive constant that may depend on the problem coefficients, but is independent of $s$, $t$, $\omega$, $P$, $\lm$ and any stopping times. We will use the elementary inequality $|xy|\leq \ep |x|^{2}+\frac{1}{4\ep}|y|^{2}$ frequently without claim.

Our aim is to solve the homogeneous stochastic control problem \eqref{state1}-\eqref{prob1}, where the coefficients
$b_t(\omega,x,v):[0,T]\times\Omega\times\R\times\Gamma\rightarrow\R$, $\sigma_t(\omega,x,v):[0,T]\times\Omega\times\R\times\Gamma\rightarrow\R^n$, $f_t(\omega,x,v):[0,T]\times\Omega\times\R\times\Gamma\rightarrow\R$ are given
$\mathcal{P}\otimes\mathcal{B}(\R)$-measurable functions, for any $v\in\Gamma$,
$g: \Omega\times\R\rightarrow\R$ is $\mathcal{F}_T\otimes\mathcal{B}(\R)$-measurable function  and they fulfill the homogeneous Assumption \ref{homocondition}. We also assume that $\varphi_t(\omega,x,v)$ is continuous w.r.t. $v\in\Gamma$ for any $(t,\omega,x)\in[0,T]\times\Omega\times\R$, where $\varphi=b,\sigma,f$.

The set of admissible  controls is defined as
\begin{align*}
\mathcal{U}:=\Big\{u: &\:[0,T]\times\Omega\rightarrow\Gamma\;\Big|\;\mbox{ $u$ is a $\mathcal{P}$ measurable process and the corresponding state }\\
&\mbox{process \eqref{state1} admits a unique continuous solution $X$ such that the cost }\\
& \mbox{functional $J(X,u)$ in \eqref{costfunction} is well-defined and finite, and $\E\Big[\sup_{t\in[0,T]}|X_t|^p\Big]<\infty$ } \Big\}.
\end{align*}

For any homogeneous function $\varphi$ of degree $p$ in $(x,u)\in\R\times \R^{m}$, we have
\begin{align}\label{homopro}
\varphi( x, u)&=|x|^p \Big(\varphi\big(1, \tfrac{u}{|x|}\big)\idd{x>0}+\varphi\big({-}1, \tfrac{u}{|x|}\big)\idd{x<0}\Big)+\varphi( 0, u)\idd{x=0}.
\end{align}
We may use this property throughout our analysis without claim.

\section{The associated BSDEs}\label{BSDE}
For $p\in\R$, $(t,v,P,\lm)\times[0,T]\times\Gamma\times\R\times\R^n$, define the following (random) functions:
\begin{align*}
&\og_{1,t}(v, P, \lm):=f_{t}(1, v)+\hf p(p-1)P |\sigma_{t}(1, v)|^2+pP b_{t}(1, v)+p\lm^{\top} \sigma_{t}(1, v),\\
&\og_{2,t}(v, P, \lm):=f_{t}({-}1, v)+\hf p(p-1)P |\sigma_{t}({-}1, v)|^2-pP b_{t}({-}1, v)-p\lm^{\top} \sigma_{t}({-}1, v),
\end{align*}
and
\begin{align*}
\og_{1,t}^*(P, \lm):=\inf_{v\in\Gamma}\og_{1,t}(v, P, \lm),\\
\og_{2,t}^*(P, \lm):=\inf_{v\in\Gamma}\og_{2,t}(v, P, \lm).
\end{align*}
Note that $\og_{1,t}^*$ and $\og_{2,t}^*$ may take the value of $-\infty$.

Now we introduce the following two one-dimensional BSDEs:
\begin{align}\label{P1}
P_t=g(1)+\int_t^T\og_{1,s}^*(P_s,\lm_s)\ds-\int_t^T\lm^{\top}_s\dw_s,~~ t\in[0,T],
\end{align}
and
\begin{align}\label{P2}
P_t=g(-1)+\int_t^T\og_{2,s}^*(P_s,\lm_s)\ds-\int_t^T\lm^{\top}_s\dw_s,~~ t\in[0,T].
\end{align}
They will play the critical role in solving the control problem \eqref{state1}-\eqref{prob1}.

\begin{defn}\label{def}
A pair of stochastic processes $(P,\lm)\in S^{\infty}_{\mathbb{F}}(0,T;\R)\times L^{2,loc}_{\mathbb{F}}(0,T;\R^n)$ is called a solution to the BSDE \eqref{P1}, if
$\int_0^T|\og_{1,t}^*(P_t,\lm_t)|\dt<\infty$ and
\begin{align*}
P_t=g(1)+\int_t^T\og_{1,s}^*(P_s,\lm_s)\ds-\int_t^T\lm^{\top}_s\dw_s,~~ t\in[0,T], \ \mathbb{P}-a.s.
\end{align*}
The solution is called nonnegative if $P\geq0$; and called uniformly positive if $P\geq c$ for some deterministic constant $c>0$. A solution to the BSDE \eqref{P2} can be defined similarly.
\end{defn}

When the coefficients in \eqref{P1} or \eqref{P2} satisfy some boundedness assumption, we have better estimates for their solutions shown as below.
\begin{lemma}\label{lemmaBMO}
Suppose $p\in\R$ and
\begin{align}\label{lemma:assu}
\int_{0}^{T}\big[|f_{t}(\pm1,0)|+|b_{t}(\pm1,0)|+ |\sigma_{s}(\pm1,0)|^{2}\big]\ds\in L^{\infty}_{\mathcal{F}_T}(\Omega;\R).
\end{align}
If $(P,\lm)$ is a solution to the BSDE \eqref{P1} or \eqref{P2}, then $\lm\in \BMO$.
\end{lemma}

\begin{proof}
Suppose $(P,\lm)$ is a solution to the BSDE \eqref{P1}.
Let $M>0$ be an upper bound of $|P|$.
Using $0\in\Gamma$ and the boundedness of coefficients, we have
\begin{align}\label{upper}
\og_{1,t}^*(P, \lm)&\leq \og_{1,t}(0, P, \lm) \nn\\
&=f_{t}(1, 0)+\hf p(p-1)P |\sigma_{t}(1, 0)|^2+pP b_{t}(1, 0)+p\lm^{\top} \sigma_{t}(1, 0)\nn\\
&\leq |f_{t}(1, 0)|+c|b_{t}(1, 0)| +c|\sigma_{t}(1,0)|^{2}+\frac{1}{5M}|\lm|^{2},
\end{align}
where $c$ is a positive constant that may depend on $M$ and $p$.

We now show $\lm\in \BMO$.
For any stopping time $\tau\leq T$, define a sequence of stopping times: $$\tau_{k}=\inf\Big\{t> \tau: \int_{\tau}^t|\lm_{s}|^2\ds>k\Big\}\wedge T.$$
Since $\lm\in L^{2}_{\mathbb{F}}(0,T;\R^n)$, we see $\tau_{k}$ increasingly approaches to $T$ as $k\to+\infty$.
Applying It\^{o}'s formula to $(P-M)^2$ and using \eqref{upper} and $0\leq M-P\leq 2M$, we get
\begin{align*}
\E\Big[\int_\tau^{\tau_{k}}|\lm|^2\ds\;\Big|\;\mathcal{F}_{\tau}\Big]
&=(P_{\tau_{k}}-M)^2-(P_{\tau}-M)^2+\E\Big[\int_{\tau}^{\tau_{k}}2(M-P)\og_{1}^*(P, \lm)\ds\;\Big|\;\mathcal{F}_{\tau}\Big].
\end{align*}
Combined with \eqref{upper} and \eqref{lemma:assu}, we get
\begin{align*}
&\quad\;\E\Big[\int_\tau^{\tau_{k}}|\lm|^2\ds\;\Big|\;\mathcal{F}_{\tau}\Big]\\
&\leq (P_{\tau_{k}}-M)^2+\E\Big[\int_{\tau}^{\tau_{k}}4M\Big(|f_{t}(1, 0)|+c|b_{t}(1, 0)| +c|\sigma_{t}(1,0)|^{2}+\frac{1}{5M}|\lm|^{2}\Big)\ds\;\Big|\;\mathcal{F}_{\tau}\Big]\\
&\leq c+\frac{4}{5}\E\Big[\int_\tau^{\tau_{k}}|\lm|^2\ds\;\Big|\;\mathcal{F}_{\tau}\Big],
\end{align*}
where $c$ is a positive constant that does not depend on $\tau$ or ${k}$.
Thanks to the definition of $\tau_{k}$, the expectations in above are finite, so
\begin{align*}
\E\Big[\int_\tau^{\tau_{k}}|\lm|^2\ds\;\Big|\;\mathcal{F}_{\tau}\Big]&\leq 5c.
\end{align*}
By sending $k\to\infty$ and using Fatou's lemma, we conclude $\lm\in\BMO$. The case for \eqref{P2} can be dealt in the same way.
\end{proof}

We now solve the BSDEs \eqref{P1} and \eqref{P2} under three different conditions.
The first condition corresponds to the standard case in LQ theory, see, e.g., \cite{HSX}, \cite{HZ}.
\begin{assmp}\label{solvablecondition1}
It holds that
\[f(\pm1,0)\in L^{\infty}_{\mathbb{F}}(0,T;\R_{+}),~g(\pm1)\in L^{\infty}_{\mathcal{F}_T}(\Omega;\R_{+}),\]
and
\[f(0,v),~f(\pm1,v)\in L^{0}_{\mathbb{F}}(0,T;\R_{+}),\]
for any $v\in\Gamma$.
Also, there are constants $L, \delta> 0$ such that
\begin{align*}
\max\{|b_{t}(\pm1,v)|^{2},~|\sigma_{t}(\pm1,v)|^{2}\}\leq \delta f_{t}(\pm1,v)+L
\end{align*}
holds for any $v\in\Gamma$ and $t\in[0,T]$.
\end{assmp}
\begin{eg}
Suppose   $\alpha>1$, $c>0$,  $\underline A$, $\overline A$,  $Q\in L^{\infty}_{\mathbb{F}}(0,T;\R)$, $\underline B$, $\overline B\in L^{\infty}_{\mathbb{F}}(0,T;\R^m)$, $\underline C$, $\overline C\in L^{\infty}_{\mathbb{F}}(0,T;\R^n)$, $\underline D$, $\overline D$, $ R\in L^{\infty}_{\mathbb{F}}(0,T;\R^{n\times m})$, $G\in L^{\infty}_{\mathcal{F}_T}(\Omega;\R)$, and $R^\top R\geq c\mathbf{1}_m$.
\begin{enumerate}
\item
If $p\geq 1$, the following functions
\begin{gather*}
 b_t(x,u)=\underline A_tx^+-\overline A_t x^-+\underline B_t^\top u^+-\overline B_t^\top u^-, \ \sigma_t(x,u)=(\underline C_tx^+-\overline C_t x^-+\underline D_t u^+-\overline D_t u^-)^\top, \\
  f_t(x,u)=(|x|^{p}+|u|^{p})(|u|^{\alpha}/|x|^{\alpha})\idd{x\neq 0}, \ g(x)=|x|^p,
\end{gather*}
satisfy Assumptions \ref{homocondition}  and \ref{solvablecondition1}.

\item If $p\geq 2$, the following functions
\begin{gather*}
 b_t(x,u)=\underline A_tx^+-\overline A_t x^-+\underline B_t^\top u^+-\overline B_t^\top u^-, \ \sigma_t(x,u)=(\underline C_tx^+-\overline C_t x^-+\underline D_t u^+-\overline D_t u^-)^\top, \\
 f_t(x,u)=|Q_tx|^{p}+|R_tu|^{p}, \ g(x)=|Gx|^p,
\end{gather*}
satisfy Assumptions \ref{homocondition} and \ref{solvablecondition1}.
 If $\underline A=\overline A$, $\underline B=\overline B$, $\underline C=\overline C$, $\underline D=\overline D$, and $p=2$, we recover the stochastic LQ problem.
\end{enumerate}


\end{eg}

\begin{thm}\label{thmcase1}
Suppose Assumptions \ref{homocondition} and \ref{solvablecondition1} hold and $p\geq1$. Then there exists a unique nonnegative solution $(P,\lm)$ to \eqref{P1} (resp. to \eqref{P2}). Moreover, $\lm\in \BMO$.
\end{thm}
\begin{proof}
We will focus on the solvability of BSDE \eqref{P1}, and the proof for \eqref{P2} is similar.

For notation simplicity, we write $w_{1,t}(v)=\max\{|b_{t}(1,v)|, |\sigma_{t}(1,v)|\}$.
Then Assumption \ref{solvablecondition1} implies $f_{t}(1, v)\geq \frac{1}{\delta}\big(w_{1,t}(v)^{2}-L\big)^{+}$. Thanks to $p\geq 1$, we have
\begin{align*}
\og_{1,t}^*(P^{+}, \lm)
&\geq \inf_{v\in\Gamma}\Big[ f_{t}(1, v)+pP^{+} b_{t}(1, v)+p\lm^{\top} \sigma_{t}(1, v)\Big] \nn\\
&\geq \inf_{v\in\Gamma}\Big[\frac{1}{\delta}\big(w_{1,t}(v)^{2}-L\big)^{+}-c(P^{+}+|\lm |)w_{1,t}(v)\Big] \nn\\
&\geq \inf_{y\in\R}\Big[\frac{1}{\delta}\big(y^{2}-L\big)^{+}-c(P^{+}+|\lm |)y\Big].
\end{align*}
Clearly,
\begin{align*}
\inf_{ y^{2}\leq 2L}\Big[\frac{1}{\delta}(y^{2}-L)^{+}-c(P^{+}+|\lm |)y\Big]
& \geq \inf_{ y^{2}\leq 2L} \Big[-c(P^{+}+|\lm |)y\Big]  \geq-c(P^{+}+|\lm |),
\end{align*}
and
\begin{align*}
\inf_{ y^{2}>2L}\Big[\frac{1}{\delta}(y^{2}-L)^{+}-c(P^{+}+|\lm |)y\Big]
& \geq \inf_{ y^{2}>2L} \Big[ \frac{1}{\delta}\Big(y^{2}-\hf y^{2}\Big)-c(P^{+}+|\lm |)y\Big] \\
& \geq \inf_{ y\in\R} \Big[ \frac{1}{2\delta}  y^{2}-c(P^{+}+|\lm |)y\Big]\\
&\geq-c(P^{+}+|\lm|)^{2},
\end{align*}
so
\begin{align}\label{lower}
\og_{1,t}^*(P^{+}, \lm) &\geq-c((P^{+}+|\lm|)^{2}+P^{+}+|\lm|).
\end{align}
On the other hand,
\begin{align}\label{upper2}
\quad \og_{1,t}^*(P^{+}, \lm)&\leq \og_{1,t}(0, P^{+}, \lm) \nn\\
&=f_{t}(1, 0)+\hf p(p-1)P^{+} |\sigma_{t}(1, 0)|^2+pP^{+} b_{t}(1, 0)+p\lm^{\top} \sigma_{t}(1, 0)\nn\\
&\leq c(1+P^{+}+|\lm|).
\end{align}
Combining the above two estimates, we have
\begin{align}\label{estimate1}
|\og_{1,t}^*(P^{+}, \lm)| \leq c(1+(P^{+})^{2}+|\lm|^2).
\end{align}

Let us consider the following linear BSDE with bounded coefficients:
\begin{align}\label{overP}
\begin{cases}
\dd\overline P=-\og_1(0,\overline P,\overline\lm)\dt+\overline \lm^{\top}\dw, \\
\overline P_T=g(1).
\end{cases}
\end{align}
Since $f_{t}(1,0)\geq0$ and $g(1)\geq0$, it admits a unique, nonnegative bounded solution $(\overline P,\overline \lm)$.

Take a constant $M>0$ an upper bound of $\overline P$.
Let $h:\R\rightarrow[0,1]$ be a smooth truncation function satisfying $h(x)=1$ for $x\in[-M,M]$, and $h(x)=0$ for $x\in({-}\infty,-2M]\cup[2M,\infty)$.
Since $h$ is compactly supported and bounded, it follows from \eqref{estimate1} that
\begin{align*}
|\og_{1,t}^*(P^{+}, \lm)|h(P)&\leq c(h(P)+(P^{+})^{2}h(P)+|\lm|^2h(P))\leq c(1+|\lm|^2).
\end{align*}
According to \cite[Theorem 2.3]{Ko}, there is a bounded, maximal solution $(\hat P,\hat \lm)$ to the following BSDE
\begin{align}\label{Ptruc}
\begin{cases}
\dd \hat P=-\og_{1}^*(\hat P^{+}, \hat\lm)h(\hat P)\dt+\hat\lm^{\top}\dw_t,\\
\hat P_T=g(1).
\end{cases}
\end{align}
Notice that $(\underline P,\underline \lm)=(0,0)$ is a solution to the following BSDE
\begin{align*}
\underline P_t=-\int_t^Tc((\underline P_{s}^{+}+|\underline\lm_{s}|)^{2}+\underline P_{s}^{+}+|\underline\lm_{s}|)h(\underline P_{s})\ds-\int_t^T\underline \lm_{s}^{\top}\dw_s,
\end{align*}
where $c$ is given in \eqref{lower}.
Then the maximal solution argument of \cite[Theorem 2.3]{Ko}, combined with \eqref{lower}, yields
\begin{align}\label{Pgeq}
\hat P\geq\underline P=0.
\end{align}
On the other hand, since $0\leq \overline P\leq M$, $h(\overline P)=1$. It follows from \eqref{overP} that
\begin{align*}
\begin{cases}
\dd\overline P=-\og_1(0,\overline P^{+},\overline\lm)h(\overline P)\dt+\overline \lm^{\top}\dw, \\
\overline P_T=g(1).
\end{cases}
\end{align*}
Applying comparison theorem to \eqref{Ptruc} and the above BSDE with Lipchitz continuous driver yields
\begin{align}\label{Pleq}
\hat P\leq \overline P\leq M.
\end{align}
Combining \eqref{Pgeq} and \eqref{Pleq}, we can assert that $h(\hat P)=1$. Together with \eqref{Ptruc} and \eqref{Pgeq}, we see $(\hat P,\hat \lm)$ is actually a solution to \eqref{P1}.
By virtue of Lemma \ref{lemmaBMO}, $\hat\lm\in\BMO$.

We are now ready to prove the uniqueness.
Suppose $(P, \lm)$ and $(\tilde P, \tilde\lm)$ are two solutions of \eqref{P1}. Then $0\leq P, \tilde P\leq M$ for some constant $M>0$, and $\lm,\tilde\lm\in\BMO$.

Define the processes
\begin{align*}
(U,V)=\left(\ln (P+a),\frac{\lm}{P+a}\right), \ (\tilde U,\tilde V)=\left(\ln (\tilde P+a),\frac{\tilde \lm}{\tilde P+a}\right),
\end{align*}
where $a>0$ is a small constant to be determined later.
Then $(U, V)$ and $(\tilde U, \tilde V)\in L^\infty_{\mathcal{F}}(0,T; \mathbb {R})\times \BMO$. Furthermore, by It\^{o}'s formula,
\begin{align*}
\label{U}
\begin{cases}
\dd U=-\Big[ \oh^*(U,V)+\frac{1}{2}V'V\Big]\dt+V'\dw,\\
U_T=\ln (g(1)+a),
\end{cases}
\end{align*}
where
\begin{align*}
\oh^*_{t}(U,V)=\inf_{v\in \Gamma}\oh_{t}(v,U,V),
\end{align*}
and
\begin{align*}
\oh_{t}(v,U,V)=&f_{t}(1,v)e^{-U}+\frac{1}{2}p(p-1)(1-ae^{-U})|\sigma_{t}(1,v)|^2+p(1-ae^{-U})b_{t}(1,v)+pV^{\top}\sigma_{t}(1,v).
\end{align*}
Noticing $U$ is bounded and using Assumption \ref{solvablecondition1}, we have
\begin{align*}
\oh_{t}(0,U,V)&\leq c(1+|V|).
\end{align*}
On the other hand, since $0\leq P\leq M$, $e^{-U}=\frac{1}{P+a}\geq\frac{1}{M+a}$ and $1-ae^{-U}=\frac{P}{P+a}\in[0,1)$. Since $p\geq 1$, it follows
\[\frac{1}{2}p(p-1)(1-ae^{-U})|\sigma_{t}(1,v)|^2\geq 0.\]
Therefore, thanks to Assumption \ref{solvablecondition1},
\begin{align*}
\oh_{t}(v,U,V)-\oh_{t}(0,U,V)
&\geq f_{t}(1,v)e^{-U}+p(1-ae^{-U})b_{t}(1,v)+pV^{\top}\sigma_{t}(1,v)-c(1+|V|)\\
&\geq \frac{ w_{1,t}(v)^{2}-L}{\delta(M+a)}-cw_{1,t}(v)(1+|V|)-c(1+|V|)\\
&\geq \frac{(w_{1,t}(v)+1)^{2}}{\delta(M+a)}-c(1+w_{1,t}(v))(1+|V|)\\
& >0,
\end{align*}
if $w_{1,t}(v)>c' (1+|V|)$ with $c'$ being a sufficiently large constant. So we conclude
\begin{align*}
\oh^*_{t}(U,V)=\inf_{w_{1,t}(v)\leq c(1+|V|)}\oh_{t}(v,U,V).
\end{align*}
Similar estimates hold if $(U,V)$ is replaced by $(\tilde U,\tilde V)$.

Set $\bar U=U-\tilde U$, $\bar V=V-\tilde V$. Then $(\bar U, \bar V)$ is a solution to the following BSDE,
\begin{align*}
\begin{cases}
\dd\bar U=-\Big[\oh^*(U,V)-\oh^*(\tilde U,\tilde V)+\frac{1}{2}(V+\tilde V)'\bar V
\Big]\dt+\bar V'\dw,\\
\bar U_T=0.
\end{cases}
\end{align*}
Applying It\^{o}'s formula to $\bar U^2$, we deduce that
\begin{align}
\label{Usquare}
\bar U_t^2&=\int_t^T\Big\{2\bar U\Big[\oh^*(U,V)-\oh^*(\tilde U,\tilde V)+\frac{1}{2}(V+\tilde V)'\bar V\Big]-\bar V'\bar V\Big\}\ds-\int_t^T2\bar U\bar V'\dw.
\end{align}

We now estimate $\bar U[\oh^*(U,V)-\oh^*(\tilde U,\tilde V)]$.
Let $c_{2}$ be a large positive constant to be chosen. Set
\begin{align*}
H_{t}(v, \tilde U, U, V)=&(f_{t}(1,v)+c_2)e^{-\tilde U}+\frac{1}{2}p(p-1)(1-ae^{-\tilde U})|\sigma_{t}(1,v)|^2-c_2e^{-U}\\
&+p(1-ae^{-U})b_{t}(1,v)
+pV^{\top}\sigma_{t}(1,v).
\end{align*}
Then we can rewrite $\oh$ in terms of $H$,
\begin{align}\label{ineqH1}
\oh^*_{t}(U, V)-\oh^*_{t}(\tilde U, \tilde V)&=\inf_{w_{1,t}(v)\leq c(1+|V|)} \oh_{t}(v, U, V)-\inf_{w_{1,t}(v)\leq c(1+|\tilde V|)}\oh_{t}(v, \tilde U, \tilde V)\nn\\
&=\inf_{w_{1,t}(v)\leq c(1+|V|)} H_{t}(v, U, U, V)-\inf_{w_{1,t}(v)\leq c(1+|\tilde V|)} H_{t}(v, \tilde U, \tilde U, \tilde V)\nn\\
&=\inf_{w_{1,t}(v)\leq c(1+|V|)} H_{t}(v, U, U, V)-\inf_{w_{1,t}(v)\leq c(1+|V|)}H_{t}(v, \tilde U, U, V)\nn\\
&\qquad\quad+\inf_{w_{1,t}(v)\leq c(1+|V|)}H_{t}(v, \tilde U, U, V)-\inf_{w_{1,t}(v)\leq c(1+|\tilde V|)} H_{t}(v, \tilde U, \tilde U, \tilde V).
\end{align}
Let's choose $a>0$ and $c_{2}$ such that
\[\frac{1}{2}p(p-1)a<\frac{1}{\delta}, ~~c_{2}>\frac{L}{\delta}.\]
Then by Assumption \ref{solvablecondition1},
\begin{align*}
f_{t}(1,v)+c_2-\frac{1}{2}p(p-1)a|\sigma_{t}(1,v)|^2&\geq \frac{1}{\delta}(w_{1,t}(v)^{2}-L)+c_2-\frac{1}{2}p(p-1)aw_{1,t}(v)^{2}>0.
\end{align*}
Consequently, for every given $(v, U, V)$, the function
\begin{align*}
\tilde U\mapsto H_{t}(v, \tilde U, U, V)
\end{align*}
is decreasing, so
\begin{align}\label{ineqL2}
\bar U\Big[\inf_{w_{1,t}(v)\leq c(1+|V|)} H_{t}(v, U, U, V)-\inf_{w_{1,t}(v)\leq c(1+|V|)}H_{t}(v, \tilde U, U, V)\Big]\leq 0.
\end{align}

On the other hand, noting that $U$ and $\tilde U$ are bounded and recalling Assumption \ref{solvablecondition1}, we have
\begin{align}\label{Hdiff}
&\quad\;\Big|\inf_{w_{1,t}(v)\leq c(1+|V|)}H_{t}(v, \tilde U, U, V)-\inf_{w_{1,t}(v)\leq c(1+|\tilde V|)} H_{t}(v, \tilde U, \tilde U, \tilde V)\Big|\nn\\
&\leq \sup_{w_{1,t}(v)\leq c(1+|V|+|\tilde V|)} \big|H_{t}(v, \tilde U, U, V)-H_{t}(v, \tilde U, \tilde U, \tilde V)\big|\nn\\
&=\sup_{w_{1,t}(v)\leq c(1+|V|+|\tilde V|)} \big|-c_2(e^{-U}-e^{-\tilde U})-pa(e^{-U}-e^{-\tilde U})b_{t}(1,v)+p\bar V^{\top}\sigma_{t}(1,v) \big|\nn\\
&\leq \sup_{w_{1,t}(v)\leq c(1+|V|+|\tilde V|)} c(|\bar U|+|\bar V|)(1+w_{1,t}(v))\nn\\
&\leq c (1+|V|+|\tilde V|)(|\bar U|+|\bar V|).
\end{align}
Therefore, we can define a process $\beta$, satisfying $|\beta|\leq c (1+|V|+|\tilde V|)$, accordingly $\beta\in\BMO$, in an obvious way such that
\begin{align}\label{ineqL3}
2\bar U\Big[\inf_{w_{1,t}(v)\leq c(1+|V|)} H_{t}(v, \tilde U, U, V)-\inf_{w_{1,t}(v)\leq c(1+|\tilde V|)}H_{t}(v, \tilde U, \tilde U, \tilde V)\Big]\leq |\beta|\bar U^2+2\beta^{\top}\bar U\bar V.
\end{align}
Moreover, notice that $U$ and $\tilde U$ are bounded and combing \eqref{ineqH1}-\eqref{ineqL3}, we have
\begin{align*}\label{ineqL4}
2\bar U\Big[\oh^*_{t}(U,V)-\oh^*_{t}(\tilde U,\tilde V)+\frac{1}{2}(V+\tilde V)'\bar V\Big]-\bar V'\bar V\leq |\beta|\bar U^2+c\beta^{\top}\bar U\bar V.
\end{align*}
Together with \eqref{Usquare}, we obtain
\begin{align*}
\bar U_t^2&\leq \int_t^T (|\beta|\bar U^2+c\beta^{\top}\bar U\bar V)\ds-\int_t^T2\bar U\bar V'\dw.
\end{align*}
Now one can repeat the proof of \cite[Theorem 3.5]{HSX} to get $(\bar U,\bar V)=(0,0)$.
This completes the proof of uniqueness.
\end{proof}

The second condition  covers two important cases: (1) the volatility coefficient $\sigma_{t}$ does not depend on the control; (2) both the drift coefficient $b(\pm1,\cdot)$ and volatility coefficient $\sigma(\pm1,\cdot)$ are bounded.
\begin{assmp}\label{solvablecondition2}
It holds that
\[f(\pm1,0)\in L^{\infty}_{\mathbb{F}}(0,T;\R_{+}),~g(\pm1)\in L^{\infty}_{\mathcal{F}_T}(\Omega;\R_{+}),\]
and
\[f(0, v),~f(\pm1,v)\in L^{0}_{\mathbb{F}}(0,T;\R_{+}), \]
for any $v\in\Gamma$.
Also, there is a constant $L\geq 0$ such that $|\sigma_{t}(\pm1,v)|\leq L$
holds for any $v\in\Gamma$ and $t\in[0,T]$;
for any constant $\ep>0$, there is a constant $L_{\ep}\geq 0$ such that
\begin{align*}
|b_{t}(\pm1,v)|\leq \ep f_{t}(\pm1,v)+L_{\ep}
\end{align*}
holds for any $v\in\Gamma$ and $t\in[0,T]$.
\end{assmp}
\begin{eg}
Suppose   $\alpha>1$,  $A$, $Q\in L^{\infty}_{\mathbb{F}}(0,T;\R)$, $B\in L^{\infty}_{\mathbb{F}}(0,T;\R^m)$, $C\in L^{\infty}_{\mathbb{F}}(0,T;\R^n)$,   $R\in L^{\infty}_{\mathbb{F}}(0,T;\R^{n\times m})$, $G\in L^{\infty}_{\mathcal{F}_T}(\Omega;\R)$, and $R^\top R\geq c\mathbf{1}_m$ for some constant $c>0$.
\begin{enumerate}[1.]
\item If $p\geq 2$, the following functions
\begin{align*}
b_t(x,u)=A_tx+B_t^\top u, \ \sigma_t(x,u)=C_tx, \ f_t(x,u)=|Q_tx|^{p}+|R_tu|^{p}, \ g(x)=|Gx|^p,
\end{align*}
satisfy Assumptions \ref{homocondition}  and \ref{solvablecondition2}.
If $\Gamma$ is symmetric, namely, $-v\in\Gamma$ whenever $v\in\Gamma$, then
\begin{align*}
\og_{1,t}^*(P, \lm)=\og_{2,t}^*(P, \lm)=\inf_{v\in\Gamma}\big[|R_tv|^p+pPB_t^\top v\big]+|Q|^p+\frac{1}{2}p(p-1)|C_t|^2+pPA_t+p\Lambda^{\top}C_t.
\end{align*}
In particular,
if  $\Gamma=\R^m$, i.e. there is no control constraint, then
\begin{align*}
\og_{1,t}^*(P, \lm)&=\og_{2,t}^*(P, \lm)=\og_{1,t}(P, \lm,\hat v_t)=\og_{2,t}(P, \lm,-\hat v_t)\\
&=(1-p)\big(B_t^{\top}(R^\top_tR_t)^{-1}B_t\big)^{\frac{p}{2(p-1)}}|P|^{\frac{p}{p-1}}
+|Q|^p+\frac{1}{2}p(p-1)|C_t|^2+pPA_t+p\Lambda^{\top}C_t,
\end{align*}
where
\begin{align*}
\hat v_t=-P|P|^{-\frac{p-2}{p-1}}\big(B_t^{\top}(R^\top_tR_t)^{-1}B_t\big)^{-\frac{p-2}{2(p-1)}}
(R^\top_tR_t)^{-1}B_t\idd{PB_t\neq0}.
\end{align*}

\item If $p\geq 1$, the following functions
\begin{gather*}
 b_t(x,u)=A_tx\sin(|u|/|x|)\idd{x\neq 0},\ \sigma_t(x,u)=C_tx\cos(|u|/|x|)\idd{x\neq 0}, \\
 f_t(x,u)=(|x|^{p}+|u|^{p})(|u|^{\alpha}/|x|^{\alpha})\idd{x\neq 0}, \ g(x)=|x|^p.
\end{gather*}
satisfy Assumptions \ref{homocondition} and \ref{solvablecondition2}.
\end{enumerate}
\end{eg}

\begin{thm}\label{thmcase2}
Suppose Assumptions \ref{homocondition} and \ref{solvablecondition2} hold and $p\geq1$.
Then there exists a unique nonnegative solution $(P,\lm)$ to \eqref{P1} (resp. to \eqref{P2}). Moreover, $\lm\in \BMO$.
\end{thm}
\begin{proof}
The proof is similar to that of Theorem \ref{thmcase1}. We only point out the main differences.
Thanks to Assumption \ref{solvablecondition2} and $p\geq 1$, we have
\begin{align}\label{lower2}
\og_{1,t}^*(P^{+}, \lm)
&\geq \inf_{v\in\Gamma}\Big[ f_{t}(1, v)+pP^{+} b_{t}(1, v)+p\lm^{\top} \sigma_{t}(1, v)\Big] \nn\nn\\
&\geq \inf_{v\in\Gamma}\Big[ f_{t}(1, v)-pP^{+} |b_{t}(1, v)|\Big]-\inf_{v\in\Gamma} f_{t}(1, v)-c|\lm| \nn\\
&=m(P^{+})-c|\lm|,
\end{align}
where we used the fact that $\inf_{v\in\Gamma} f_{t}(1, v)\geq 0$ and defined
\[m(x):=\inf_{v\in\Gamma}\Big[ f_{t}(1, v)-px |b_{t}(1, v)|\Big]-\inf_{v\in\Gamma} f_{t}(1, v),~~x\in\R.\]
We now show $m$ is a real-valued function on $({-}c,c)$ for any $c>0$.
Indeed, by taking $\ep=1/(pc)$ in Assumption \ref{solvablecondition2}, one has, for $x\in({-}c,c)$,
\begin{align*}
m(x) &\geq\inf_{v\in\Gamma}\Big[ f_{t}(1, v)-pc \Big(f_{t}(1, v)/(pc)+L_{1/(pc)}\Big)\Big]-f_{t}(1, 0)\\
&=-pcL_{1/(pc)}-f_{t}(1, 0)>-\infty.
\end{align*}
Meanwhile,
\begin{align*}
m(x)&\leq f_{t}(1, 0)-px |b_{t}(1, 0)|
\leq f_{t}(1, 0)+p|x| (f_{t}(1, 0)+L_{1})<\infty.
\end{align*}
Therefore, $m$ is a real-valued function on $\R$. It is also trivial to see that $m$ is non-increasing and concave (thus continuous) on $\R$ with $m(0)=0$. One can replace \eqref{lower} by the above lower bound \eqref{lower2} to show the existence.

To establish the uniqueness,
we choose constants $a,\ep>0$ such that $0<\ep p(1-ae^{-U})<\hf e^{-U}$. Then
\begin{align*}
\oh_{t}(v,U,V)-\oh_{t}(0,U,V)
&\geq f_{t}(1,v)e^{-U}+p(1-ae^{-U})b_{t}(1,v)+pV^{\top}\sigma_{t}(1,v)-c(1+|V|)\\
&\geq f_{t}(1,v)e^{-U}-p(1-ae^{-U}) (\ep f_{t}(1, v)+L_{\ep})-c(1+|V|)\\
&\geq \frac{1}{2}f_{t}(1,v)e^{-U}-c(1+|V|)>0
\end{align*}
if $f_{t}(1,v)\geq c'(1+|V|)$ with $c'$ being a sufficiently large constant. Hence,
\begin{align*}
\oh^*_{t}(U,V)=\inf_{f_{t}(1,v)\leq c(1+|V|)}\oh_{t}(v,U,V).
\end{align*}
Now we can replace \eqref{Hdiff} by
\begin{align*}
&\quad\;\Big|\inf_{f_{t}(1,v)\leq c(1+|V|)}H_{t}(v, \tilde U, U, V)-\inf_{f_{t}(1,v)\leq c(1+|\tilde V|)} H_{t}(v, \tilde U, \tilde U, \tilde V)\Big|\nn\\
&\leq \sup_{f_{t}(1,v)\leq c(1+|V|+|\tilde V|)} \big|H_{t}(v, \tilde U, U, V)-H_{t}(v, \tilde U, \tilde U, \tilde V)\big|\nn\\
&=\sup_{f_{t}(1,v)\leq c(1+|V|+|\tilde V|)} \big|-c_2(e^{-U}-e^{-\tilde U})-pa(e^{-U}-e^{-\tilde U})b_{t}(1,v)+p\bar V^{\top}\sigma_{t}(1,v) \big|\nn\\
&\leq c (1+|V|+|\tilde V|)(|\bar U|+|\bar V|),
\end{align*}
where we used $|b_{t}(\pm1,v)|\leq f_{t}(\pm1,v)+L_{1}$ to get the last inequality.
\end{proof}

Finally, we introduce the third condition, which corresponds to the singular case in LQ theory, see, e.g., \cite{HSX}, \cite{HZ}.
\begin{assmp}\label{solvablecondition3}
It holds that
\[f(\pm1,0)\in L^{\infty}_{\mathbb{F}}(0,T;\R_{+}),~\sigma(\pm1,0)\in L^{\infty}_{\mathbb{F}}(0,T;\R^{n}),~g(\pm1)\in L^{\infty}_{\mathcal{F}_T}(\Omega;\R_{+}),\]
and
\[f(0, v),~f(\pm1,v)\in L^{0}_{\mathbb{F}}(0,T;\R_{+}), \]
for any $v\in\Gamma$.
Also, there are constants $L$, $\delta> 0$ and $\eta>0$ such that
\begin{align}\label{delta}
|b_{t}(\pm1,v)|\leq \delta|\sigma_{t}(\pm1,v)|^{2}+L, ~~g(\pm1)\geq \eta,
\end{align}
holds for any $v\in\Gamma$ and $t\in[0,T]$.
\end{assmp}
\begin{eg}
Suppose $p>1$, $\underline A$, $\overline A$, $ Q$, $\tilde Q\in L^{\infty}_{\mathbb{F}}(0$, $T;\R)$, $\underline B$, $\overline B\in L^{\infty}_{\mathbb{F}}(0$, $T;\R^m)$, $\underline C$, $\overline C\in L^{\infty}_{\mathbb{F}}(0$, $T;\R^n)$, $D$, $R$, $\tilde R\in L^{\infty}_{\mathbb{F}}(0$, $T;\R^{n\times m})$, $G$, $\tilde G\in L^{\infty}_{\mathcal{F}_T}(\Omega;\R)$$, $ and $D^\top D\geq c\mathbf{1}_m$, $G\geq c$, $\tilde G\geq c$ for some constant $c>0$. Then, the following functions
\begin{gather*}
 b_t(x,u)=\underline A_tx^+-\overline A_t x^-+\underline B_t^\top u^+-\overline B_t^\top u^-, \ \sigma_t(x,u)=(\underline C_tx^+-\overline C_t x^-+ D_t u)^{\top}, \\
 f_t(x,u)=|Q_tx|^{p}+|R_tu|^{p}+|\tilde Q_tx|^{p/3}|\tilde R_tu|^{2p/3}, \ g(x)=G(x^+)^p+\tilde G(x^-)^p,
\end{gather*}
satisfy Assumptions \ref{homocondition} and \ref{solvablecondition3}.

\end{eg}

\begin{thm}\label{thmcase3}
Suppose Assumptions \ref{homocondition} and \ref{solvablecondition3} hold and $p>1+2\delta $. Then there exists a unique uniformly positive solution $(P,\lm)$ to \eqref{P1} (resp. to \eqref{P2}). Moreover, $\lm\in \BMO$.
\end{thm}
\begin{proof}
Since $p>1+2\delta$, there is a constant $\ep>0$ such that $p-1-2\delta-2\ep>0$. For $P>0$, we have
\[\lm^{\top} \sigma_{t}(1, v)\leq \ep P |\sigma_{t}(1, v)|^{2}+\frac{|\lm |^{2}}{4\ep P},\]
so it follows from Assumption \ref{solvablecondition3} that
\begin{align}\label{Hlowersin}
\og_{1,t}^*(P, \lm)
&= \inf_{v\in\Gamma}\Big[ \hf p(p-1)P |\sigma_{t}(1, v)|^2+f_{t}(1, v)+pP b_{t}(1, v)+p\lm^{\top} \sigma_{t}(1, v)\Big]\nn\\
&\geq \inf_{v\in\Gamma}\Big[ \hf p(p-1)P |\sigma_{t}(1, v)|^2-p P(\delta|\sigma_{t}(1,v)|^{2}+L)-\ep p P |\sigma_{t}(1, v)|^{2}-\frac{p|\lm |^{2}}{4\ep P}\Big]\nn\\
&=\inf_{v\in\Gamma} \hf(p-1-2\delta-2\ep)|\sigma_{t}(1, v)|^2pP-LpP-\frac{p|\lm |^{2}}{4\ep P}\nn\\
&\geq-LpP-\frac{p|\lm |^{2}}{4\ep P}\nn\\
&\geq-c(P+|\lm|^{2}/P),
\end{align}
where we used $p-1-2\delta-2\ep$, $p$, $P>0$ to get the the second inequality.

Let $c$ be the constant given in \eqref{Hlowersin}. Let $h:\R\rightarrow[0,1]$ be a smooth truncation function satisfying $h(x)=0$ for $x\in({-}\infty,\frac{1}{2}\eta e^{-cT}]$, and $h(x)=1$ for $x\geq \eta e^{-cT}$.
Noting \eqref{upper2} and \eqref{Hlowersin},
according to \cite[Theorem 2.3]{Ko}, there exists a bounded maximal solution $(\hat P, \hat \lm)$ to the following BSDE
\begin{align}
\label{P5}
\begin{cases}
\dd \hat P=-\og_1^*(\hat P,\hat \lm)h(\hat P) \dt+\hat \lm'\dw,\\
\hat P_T=g(1).
\end{cases}
\end{align}
Notice that $(\underline P, \underline\lm)=(\eta e^{-c(T-t)}, 0)$ satisfies the following BSDE:
\begin{align*}
\begin{cases}
\dd \underline P=c(\underline P+|\underline\lm|^{2}/\underline P)h(\underline P)\dt+\underline\lm'\dw,\\
\underline P_T=\eta.
\end{cases}
\end{align*}
Then the maximal argument in \cite[Theorem 2.3]{Ko} applied to the above two BSDEs yields
\[
\hat P\geq \underline P\geq \eta e^{-cT}.
\]
It hence follows $h(\hat P)=1$.
Together with \eqref{P5}, we see $(\hat P, \hat \lm)$ is actually a uniformly positive solution to \eqref{P1}.

To establish the uniqueness, notice $\hat P$ is uniformly positive, so we can take $a=0$ in the proof of Theorem \ref{thmcase1}.
  Assumption \ref{solvablecondition3} leads to
\begin{align*}
|b_{t}(\pm1,v)|\leq |\sigma_{t}(\pm1,v)|^{2}(\delta+\ep)
\end{align*}
when $|\sigma_{t}(\pm1,v)|^{2}\geq L/\ep$.
Therefore,
\begin{align*}
&\quad\;\oh_{t}(v,U,V)-\oh_{t}(0,U,V)\\
&\geq f_{t}(1,v)e^{-U}+\frac{1}{2}p(p-1)|\sigma_{t}(1,v)|^2+pb_{t}(1,v)+pV^{\top}\sigma_{t}(1,v)-c(1+|V|)\\
&\geq \hf (p-1-2\delta-2\ep) p|\sigma_{t}(1,v)|^2-pL-p|V||\sigma_{t}(1,v)|-c(1+|V|)>0,
\end{align*}
if $|\sigma_{t}(1,v)|>c'(1+|V|)$ with $c'$ being a sufficiently large constant.
Hence,
\begin{align*}
\oh^*_{t}(U,V)=\inf_{|\sigma_{t}(1,v)|\leq c(1+|V|)}\oh_{t}(v,U,V).
\end{align*}
Recall that $a=0$, so we can replace \eqref{Hdiff} by the estimate
\begin{align*}
&\quad\;\Big|\inf_{|\sigma_{t}(1,v)|\leq c(1+|V|)}H_{t}(v, \tilde U, U, V)-\inf_{|\sigma_{t}(1,v)|\leq c(1+|\tilde V|)} H_{t}(v, \tilde U, \tilde U, \tilde V)\Big|\nn\\
&\leq \sup_{|\sigma_{t}(1,v)|\leq c(1+|V|+|\tilde V|)} \big|H_{t}(v, \tilde U, U, V)-H_{t}(v, \tilde U, \tilde U, \tilde V)\big|\nn\\
&=\sup_{|\sigma_{t}(1,v)|\leq c(1+|V|+|\tilde V|)} \big|-c_2(e^{-U}-e^{-\tilde U})+p\bar V^{\top}\sigma_{t}(1,v) \big|\nn\\
&\leq c (1+|V|+|\tilde V|)(|\bar U|+|\bar V|).
\end{align*}
Hence, the estimate \eqref{Hdiff} still holds and the reminder argument is similar as before.
\end{proof}

\section{Verification: solution to the  problem \eqref{state1}-\eqref{prob1}}\label{Verification}
In this section, we will solve the stochastic control problem \eqref{state1}-\eqref{prob1} by providing an optimal control expressed by the solutions to BSDEs \eqref{P1} and \eqref{P2} obtained in the previous section.

\begin{thm}\label{control}
Suppose  one of the following cases holds:
\begin{description}
\item[\qquad Case I] Assumptions \ref{homocondition} and \ref{solvablecondition1} hold and $p>1$;
\item[\qquad Case II] Assumptions \ref{homocondition} and \ref{solvablecondition2} hold and $p>1$;
\item[\qquad Case III] Assumptions \ref{homocondition} and \ref{solvablecondition3} hold and $p>1+2\delta$.
\end{description}
Then the problem \eqref{state1}-\eqref{prob1} admits an optimal feedback control of time $t$ and state $X$:
\begin{align}
\label{opticon}
u^*(t,X)=\hat v_{1,t}X^++\hat v_{2,t}X^-;
\end{align}
with the corresponding optimal value:
\begin{align*}
\inf_{u\in\mathcal{U}}\; J(x,u)=P_{1,0}(x^+)^p+P_{2,0}(x^-)^p,
\end{align*}
where $(P_1, \lm_1)$ and $(P_2, \lm_2)$ are the unique nonnegative solutions (in Case I and Case II) and uniformly positive solutions (in Case III) to \eqref{P1} and \eqref{P2}, respectively;
and $\hat v_{1}$ and $\hat v_{2}$ are any $\mathcal{P}$-measurable processes such that
\begin{align*}
\hat v_{1,t}\in\Gamma,~~\og_{1,t}(\hat v_{1,t}, P_{1,t}, \lm_{1,t})=\og_{1,t}^*(P_{1,t}, \lm_{1,t}),\\
\hat v_{2,t}\in\Gamma,~~\og_{2,t}(\hat v_{2,t}, P_{2,t}, \lm_{2,t})=\og_{2,t}^*(P_{2,t}, \lm_{2,t}).
\end{align*}
\end{thm}
\begin{proof}
In any case, we know the existence and uniqueness of $(P_1, \lm_1)$ and $(P_2, \lm_2)$ by Theorems \ref{thmcase1}, \ref{thmcase2} or \ref{thmcase3}.
In particular, we have the finiteness of $\og_{1,t}^*(P_{1,t}, \lm_{1,t})$ and $\og_{2,t}^*(P_{2,t}, \lm_{2,t})$.
Thanks to the homogeneity, we also have $f_{t}\geq0$ in all cases.

\textbf{Step 1.}
We first show the existence of $\hat v_{1}$ and $\hat v_{2}$ and provide some estimates for them. In Case I, we have, recalling $w_{1,t}(v)=\max\{|b_{t}(1,v)|, |\sigma_{t}(1,v)|\}$,
\begin{align*}
\og_{1}(v, P_1,\lm_1)-\og_{1}^*(P_1,\lm_1)
&\geq\og_{1}(v, P_1,\lm_1)-\og_{1}(0, P_1,\lm_1)\\
&\geq f_{t}(1, v)+pP_1 b_{t}(1, v)+p\lm_1^{\top} \sigma_{t}(1, v)-c(1+|\lm_{1}|)\nn\\
&\geq \frac{1}{\delta}(w_{1,t}(v)^{2}-L)-c(1+|\lm_{1} |)(1+w_{1,t}(v))> 0,
\end{align*}
if $w_{1,t}(v)>c'(1+|\lm_{1}|)$ with $c'$ being a sufficiently large constant.
Similarly, we have $\og_{2}(v, P_2,\lm_2)-\og_{2}^*(P_2,\lm_2)>0$ if $w_{2,t}(v)>c'(1+|\lm_{2}|)$.
Thanks to the finiteness and continuity of $\og_{1}$ and $\og_{2}$ w.r.t. $v$, we see that $\hat v_{1}$ and $\hat v_{2}$ exist and are real-valued and satisfy
\begin{align*}
w_{1,t}(\hat v_{1})\leq c(1+|\lm_1|),~~w_{2,t}(\hat v_{2})\leq c(1+|\lm_2|),
\end{align*}
so
\begin{align}\label{sigbup1}
|b_{t}(1,\hat v_{1})|+|\sigma_{t}(1,\hat v_{1})|\leq c(1+|\lm_1|), ~~
|b_{t}(-1,\hat v_{2})|+|\sigma_{t}(-1,\hat v_{2})|\leq c(1+|\lm_2|).
\end{align}
The existence and finiteness of $\hat v_{1}$ and $\hat v_{2}$ in Cases II and III can be established similarly.
In Case II, one can show that
\begin{align}\label{hatvupp2}
f_{t}(1,\hat v_{1})\leq c(1+|\lm_1|), ~~
f_{t}(-1,\hat v_{2})\leq c(1+|\lm_2|),
\end{align}
which together with Assumption \ref{solvablecondition2} implies
\begin{align}\label{sigbup1case2}
|b_{t}(1,\hat v_{1})|\leq c(1+|\lm_1|), ~~
|b_{t}(-1,\hat v_{2})|\leq c(1+|\lm_2|), ~~
|\sigma_{t}(1,\hat v_{1})|+|\sigma_{t}(-1,\hat v_{2})|\leq c.
\end{align}
In Case III, one can show that
\begin{align}\label{sigup1}
|\sigma_{t}(1, \hat v_{1})|\leq c(1+|\lm_1|),
~~|\sigma_{t}(-1, \hat v_{2})|\leq c(1+|\lm_2|),
\end{align}
which together with Assumption \ref{solvablecondition3} implies
\begin{align}\label{bup1}
|b_{t}(1, \hat v_{1})|\leq c(1+|\lm_1|^{2}),
~~|b_{t}(-1, \hat v_{2})|\leq c(1+|\lm_2|^{2}).
\end{align}

\textbf{Step 2.} Next we establish the following lower bound
\begin{align}\label{localoptimal}
J(X, u)=\E\Big[\int_0^{T} f_{t}(X_t, u_t)\dt+g(X_{T})\Big] &\geq P_{1,0}(x^+)^p+P_{2,0}(x^-)^p,
\end{align}
for any admissible control $u\in\mathcal{U}$.

Note we have $p>1$ in all the three cases, so for any admissible control $u\in\mathcal{U}$, by \cite[Lemma 2.2]{BDHPS},
\begin{align*}
\dd\: (X^{+}_{t})^{p}&=p(X^{+}_{t})^{p-1}(b_{t}(X_t, u_t)\dt+\sigma_{t}(X_t, u_t)^{\top}\dw_t)\\
&\qquad\qquad\;+\hf p(p-1)(X^{+}_{t})^{p-2}\idd{X_t>0}|\sigma_{t}(X_t, u_t)|^2\dt,\\
-\dd\: (X^{-}_{t})^{p}&=p(X^{-}_{t})^{p-1}(b_{t}(X_t, u_t)\dt+\sigma_{t}(X_t, u_t)^{\top}\dw_t)\\
&\qquad\qquad\;+\hf p(p-1)(X^{-}_{t})^{p-2}\idd{X_t<0}|\sigma_{t}(X_t, u_t)|^2\dt.
\end{align*}
Applying \ito to $P_{1,t}(X^{+}_{t})^{p}$ and $P_{2,t}(X^{-}_{t})^{p}$ respectively, we get
\begin{align*}
\dd\: (P_{1,t}(X^{+}_{t})^{p})
&=\idd{X_t>0}\Big[\hf p(p-1)P_{1,t}(X^{+}_{t})^{p-2}|\sigma_{t}(X_t, u_t)|^2-(X^{+}_{t})^{p}\og^{*}_{1,t}(P_{1,t}, \lm_{1,t})\\
&\qquad\qquad\;+p(X_t^+)^{p-1}(P_{1,t}b_{t}(X_t,u_t)+\lm_{1,t}^{\top}\sigma_{t}(X_t,u_t))\Big]\dt\\
&\qquad\;+\Big[(X_t^+)^{p}\lm_{1,t}+pP_{1,t}(X_t^+)^{p-1}\sigma_{t}(X_t,u_t)\Big]^{\top}\dw_t,
\end{align*}
and
\begin{align*}
\dd\: (P_{2,t}(X^{-}_{t})^{p})
&=\idd{X_t<0}\Big[\hf p(p-1)P_{2,t}(X^{-}_{t})^{p-2}|\sigma_{t}(X_t, u_t)|^2-(X^{-}_{t})^{p}\og^{*}_{2,t}(P_{2,t}, \lm_{2,t})\\
&\qquad\qquad\;-p(X_t^-)^{p-1}(P_{2,t}b_{t}(X_t,u_t)+\lm_{2,t}^{\top}\sigma_{t}(X_t,u_t))\Big]\dt\\
&\qquad\;+\Big[(X_t^-)^{p}\lm_{2,t}-pP_{2,t}(X_t^-)^{p-1}\sigma_{t}(X_t,u_t)\Big]^{\top}\dw_t.
\end{align*}
By standard construction, there is a sequence of increasing stopping times $\{\tau_{k}\}$ with $\lim_{k}\tau_{k}=T$ such that, after summing the above two equations,
\begin{align}\label{sum}
&\quad\;\E\Big[\int_0^{\tau_{k}} f_t(X_t, u_t)\dt+P_{1,\tau_{k}}(X^{-}_{\tau_{k}})^{p}+P_{2,\tau_{k}}(X^{-}_{\tau_{k}})^{p}\Big]\nn\\
&=P_{1,0}(x^+)^p+P_{2,0}(x^-)^p+\mathbb{E}\int_0^{\tau_{k}}\varphi_{t}(X_t,u_t)\dt,
\end{align}
where
\begin{align*}
\varphi_{t}(x,u)=f_{t}(x,u) &+\idd{x>0}\hf p(p-1)P_{1,t}(x^{+})^{p-2}|\sigma_{t}(x, u)|^2-(x^{+})^{p}\og^{*}_{1,t}(P_{1,t}, \lm_{1,t})\\
&\qquad\quad\;+p(x^+)^{p-1}(P_{1,t} b_{t}(x,u)+\lm_{1,t}^{\top}\sigma_{t}(x,u))\\
&+\idd{x<0}\hf p(p-1)P_{2,t}(x^{-})^{p-2}|\sigma_{t}(x, u)|^2-(x^{-})^{p}\og^{*}_{2,t}(P_{2,t}, \lm_{2,t})\\
&\qquad\quad\;-p(x^-)^{p-1}(P_{2,t}b_{t}(x,u)+\lm_{2,t}^{\top}\sigma_{t}(x,u)).
\end{align*}

We now show that $\varphi_{t}(x,u)\geq0$ for any $x\in\R$, $u\in\Gamma$ and $t\in[0,T]$. Since $u\in\Gamma$ and $\Gamma$ is a cone,
$v:=\frac{u}{|x|}\idd{x\neq 0}\in\Gamma$. Recall that  $f_{t}$ is homogeneous of degree $p>1$ and $b_{t}$, $\sigma_{t}$ are homogeneous of degree 1 in $(x,u)$, so
\begin{align*}
\varphi_{t}(x,u)&=\idd{x>0}|x|^{p}\Big(\hf p(p-1)P_{1,t}|\sigma_{t}(1, v)|^2-\og^{*}_{1,t}(P_{1,t}, \lm_{1,t})+p(P_{1,t}b_{t}(1,v)+\lm_{1,t}^{\top}\sigma_{t}(1,v))+f_{t}(1,v)\Big)\\
&\quad+\idd{x<0}|x|^{p}\Big(\hf p(p-1)P_{2,t}|\sigma_{t}({-}1, v)|^2-\og^{*}_{2,t}(P_{2,t}, \lm_{2,t})\\
&\qquad\qquad\qquad\quad-p(P_{2,t}b_{t}({-}1,v)+\lm_{2,t}^{\top}\sigma_{t}({-}1,v))+f_{t}({-}1,v)\Big)\\
&\quad\;+\idd{x=0}f_t(0,u)\\
&\geq0,
\end{align*}
by the definitions of $\og^*_1$, $\og^*_2$ and $f_{t}(0,u)\geq0$.
Now it follows from \eqref{sum} that
\begin{align*}
\E\Big[\int_0^{\tau_{k}} f_{t}(X_t, u_t)\dt+P_{1,\tau_{k}}(X^{+}_{\tau_{k}})^{p}+P_{2,\tau_{k}}(X^{-}_{\tau_{k}})^{p}\Big]
&\geq P_{1,0}(x^+)^p+P_{2,0}(x^-)^p.
\end{align*}
Since $f_{t}\geq 0$, we can apply the monotone convergence theorem to the integral  above; and since $\E\Big[\sup_{t\in[0,T]}|X_t|^p\Big]<\infty$ and $P_{1}$ and $P_{2}$ are bounded, we can apply the dominated convergence theorem to the other two terms to get
\begin{align*}
\E\Big[\int_0^{T} f_{t}(X_t, u_t)\dt+P_{1,T}(X^{+}_{T})^{p}+P_{2,T}(X^{-}_{T})^{p}\Big]
&\geq P_{1,0}(x^+)^p+P_{2,0}(x^-)^p.
\end{align*}
Since $P_{1,T}=g(1)$, $P_{2,T}=g({-}1)$ and $g$ is homogenous, the lower bound \eqref{localoptimal} is established.

\textbf{Step 3.}
We next show the following claims hold:
\begin{description}
\item[\quad Claim (i)] The state process \eqref{state1} admits a unique continuous solution, denoted by $X^*$, under the feedback control $u^{*}$ given in \eqref{opticon};

\item[\quad Claim (ii)] $J(X^*,u^*)=\E\Big[\int_0^T f_{t}(X^*_t, u_t^*)\dt+g(X^*_{T})\Big]\leq P_{1,0}(x^+)^p+P_{2,0}(x^-)^p$;

\item[\quad Claim (iii)] $\E\Big[\sup_{t\in[0,T]}|X_t^*|^p\Big]<\infty$.
\end{description}
Claims (i)-(iii) show that $u^{*}$ is an admissible control in $\mathcal{U}$ to the problem \eqref{state1}-\eqref{prob1}. Furthermore, Claim (ii) and \eqref{localoptimal} indicate that $u^{*}$ is indeed an optimal control to \eqref{state1}-\eqref{prob1}. Therefore, the proof of the theorem will be complete if
Claims (i)-(iii) can be proved. We now prove them one by one.
\bigskip

\textbf{Proof of Claim (i).}
Substituting the feedback control $u^{*}$ into the state process \eqref{state1}, and using the homogeneous property of $b_{t}$ and $\sigma_{t}$, we have
\begin{align}
\label{SDEstate}
\begin{cases}
\dd X_t=b_{t}(X_t,u^*(t,X_{t}))\dt+\sigma_{t}(X_t,u^*(t,X_t))^{\top}\dw_t\\
\qquad=\idd{X_{t}>0}X_t [b_{t}(1, \hat v_1)\dt+\sigma_{t}(1, \hat v_1)^{\top}\dw_t]+\idd{X_{t}<0}X_t [b_{t}({-}1, \hat v_2)\dt+\sigma_{t}({-}1, \hat v_2)^{\top}\dw_t], \\
X_0=x.
\end{cases}
\end{align}
By the estimates \eqref{sigbup1}-\eqref{bup1} and $\lm_1, \lm_2\in \BMO$,
according to Gal'chuk \cite[basic theorem on pp. 756-757]{Ga}, the SDE \eqref{SDEstate} admits a unique continuous solution,
which is clearly given by
\begin{align}\label{optimalx}
X^*_t=\begin{cases}
X^*_{1,t},&\quad\mbox{if $x>0$;}\\
0,&\quad\mbox{if $x=0$;}\\
X^*_{2,t},&\quad\mbox{if $x<0$;}
\end{cases}
\end{align}
where
\begin{align*}
&X^*_{1,t}=x\exp\bigg\{\int_0^t\Big(b_{s}(1, \hat v_1)-\frac{1}{2}|\sigma_{s}(1, \hat v_1)|^2\Big)\ds+\int_0^t\sigma_{s}(1, \hat v_1)^{\top}\dw_s \bigg\},\\
&X^*_{2,t}=x\exp\bigg\{\int_0^t\Big(b_{s}({-}1, \hat v_2)-\frac{1}{2}|\sigma_{s}({-}1, \hat v_2)|^2\Big)\ds+\int_0^t\sigma_{s}({-}1, \hat v_2)^{\top}\dw_s \bigg\}.
\end{align*}
This completes the proof of Claim (i).
\bigskip

\textbf{Proof of Claim (ii).}
Thanks to the continuity of $X^*$, the estimates \eqref{sigbup1}-\eqref{bup1} and $\varphi_{t}(X^*_t,u^*_t)=0$, we know,  recalling \eqref{sum}, there is a
sequence of increasing stopping times $\{\tau_{k}\}$ with $\lim_{k}\tau_{k}=T$ such that,
\begin{align*}\label{stop}
\mathbb{E}\Bigg[\int_0^{\iota\wedge\tau_k}f_{s}(X_s^*,u^*_{s})\ds+P_{1,\iota\wedge\tau_k}(X_{\iota\wedge\tau_k}^{*,+})^p
+P_{2,\iota\wedge\tau_k}(X_{\iota\wedge\tau_k}^{*,-})^p
\Bigg]
&=P_{1,0}(x^+)^p+P_{2,0}(x^-)^p,
\end{align*}
for any stopping time $\iota$ valued in $[0,T]$.
Now applying Fatou's lemma to the left hand side of above leads to
\begin{equation}
\label{ustar}
\E\Big[\int_0^{\iota} f_{t}(X_t^*, u_t^*)\dt+P_{1,\iota}(X^{*,+}_{\iota})^{p}+P_{2,\iota}(X^{*,-}_{\iota})^{p}\Big]\leq P_{1,0}(x^+)^p+P_{2,0}(x^-)^p.
\end{equation}
Especially, we deduce Claim (ii) by taking $\iota=T$ in above.

Furthermore, in Case III, we have $f_{t}\geq 0$ and $P_{1} $ and $ P_{2}$ are uniformly positive, so \eqref{ustar} implies
\begin{equation}
\label{uniboundstopping}
\sup_{\iota}\E\Big[|X^{*}_{\iota}|^{p}\Big]\leq K~~\mbox{in Case III}
\end{equation}
with some constant $K$, where the supreme is taken over all stopping times $\iota$ valued in $[0,T]$.
\bigskip

\textbf{Proof of Claim (iii).}
We now prove $$\E\Big[\sup_{t\in[0,T]}|X_t^*|^p\Big]<\infty.$$
This is trivial by \eqref{optimalx} when $x=0$. We only deal with the case $x>0$ since the case $x<0$ can be conducted similarly.

\underline{Case I.}
Thanks to Assumptions \ref{homocondition} and \ref{solvablecondition1}, we have
\begin{align*}
|X_s^*|^{p-1}|b_{s}(X_s^*,u_s^*)|&=|X_s^*|^{p}|b_{s}(1,\hat v_{1})|\leq |X_s^*|^{p} (\delta f_{s}(1,\hat v_{1})+L)^{1/2} \\
&\leq \frac{1}{2} |X_s^*|^{p} (\delta f_{s}(1,\hat v_{1})+L+1)= \frac{1}{2}(\delta f_{s}(X_s^*,u_s^*)+(L+1)|X_s^*|^{p} ),
\end{align*}
and
\begin{align*}
|X_s^*|^{p-2}\idd{X_s^*\neq0}|\sigma_{s}(X_s^*,u_s^*)|^{2}
&=|X_s^*|^{p}\idd{X_s^*\neq0}|\sigma_{s}(1,\hat v_{1})|^{2}\\
&\leq|X_s^*|^{p}\idd{X_s^*\neq0} (\delta f_{s}(1,\hat v_{1})+L)= \delta f_{s}(X_s^*,u_s^*)+L|X_s^*|^{p}.
\end{align*}
Let $a$ be any real constant.
Applying \ito \cite[Lemma 2.2]{BDHPS} to $e^{at}|X_t^{*}|^p$ gives
\begin{align*}
e^{at}|X_t^{*}|^p& =x^p+\int_0^{t}e^{as}
\bigg(a|X_s^{*}|^p+p|X_s^*|^{p-1}\sgn{X_s^*}b_{s}(X_s^*,u_s^*)+\frac{1}{2}p(p-1)|X_s^*|^{p-2}\idd{X_s^*\neq0}|\sigma_{s}(X_s^*,u_s^*)|^2\bigg)\ds\nn\\
&\qquad\;+\int_0^{t}e^{as}p|X_s^*|^{p-1}\sgn{X_s^*}\sigma_{s}(X_s^*,u_s^*)^{\top}\dw_s\\
&\leq x^p+\int_0^{t}e^{as}
\bigg(a|X_s^{*}|^p+\frac{1}{2}p(\delta f_{s}(X_s^*,u_s^*)+(L+1)|X_s^*|^{p} )+\frac{1}{2}p(p-1)(\delta f_{s}(X_s^*,u_s^*)+L|X_s^*|^{p})\bigg)\ds\nn\\
&\qquad\;+\int_0^{t}e^{as}p|X_s^*|^{p-1}\sgn{X_s^*}\sigma_{s}(X_s^*,u_s^*)^{\top}\dw_s.\end{align*}
Construct a sequence of stopping times as follows:
\begin{align*}
\theta_{k}:=\inf\Big\{&t\geq0\;\Big|\;|X_t^*|
> k \Big\}\wedge T.
\end{align*}
Since $X^*$ is continuous, $\theta_{k}$ increasingly converges to $T$ as $k\to\infty$.
By the Burkholder-Davis-Gundy  inequality,
\begin{align}\label{BDG1}
&\quad\; \E\bigg[\sup_{t\in[0,\theta_{k}]} \bigg|\int_0^{t}e^{as}p|X_s^*|^{p-1}\sgn{X_s^*}\sigma_{s}(X_s^*,u_s^*)^{\top}\dw_s\bigg|\bigg]\nn\\
&\leq M\E\bigg[\bigg(\int_0^{\theta_{k}}e^{2as}|X_s^*|^{2(p-1)}|\sigma_{s}(X_s^*,u_s^*)|^2\ds\bigg)^{1/2}\bigg]\nn\\
&\leq M\E\bigg[\sup_{t\in[0,\theta_{k}]}e^{at/2}|X_t^{*}|^{p/2} \bigg(\int_0^{\theta_{k}} e^{as}|X_s^*|^{p-2}\idd{X^*_s\neq0}|\sigma_{s}(X_s^*,u_s^*)|^2\ds\bigg)^{1/2}\bigg]\nn\\
&\leq \hf\E\bigg[\sup_{t\in[0,\theta_{k}]}e^{at}|X_t^{*}|^{p}\bigg]
+M^{2}\E\bigg[\int_0^{\theta_{k}} e^{as}|X_s^*|^{p-2}\idd{X^*_s\neq0}|\sigma_{s}(X_s^*,u_s^*)|^2\ds\bigg]\nn\\
&\leq \hf\E\bigg[\sup_{t\in[0,\theta_{k}]}e^{at}|X_t^{*}|^{p}\bigg]
+M^{2}\E\bigg[\int_0^{\theta_{k}} e^{as}\Big( \delta f_{s}(X_s^*,u_s^*)+L|X_s^*|^{p}\Big)\ds\bigg],
\end{align}
where $M$ is a constant independent of $k$ and $a$.
Combining the above two estimates, we conclude, for $a\leq -\frac{1}{2}p^{2} L-\frac{1}{2}p-M^{2}L$,
\begin{align*}
\hf\E\bigg[\sup_{t\in[0,\theta_{k}]}e^{at}|X_t^{*}|^{p}\bigg]
&\leq x^p+\E\bigg[\sup_{t\in[0,\theta_{k}]}\int_0^{t}e^{as}
\bigg(a|X_s^{*}|^p+\frac{1}{2}p(\delta f_{s}(X_s^*,u_s^*)+(L+1)|X_s^*|^{p} )\nn\\
&\qquad\qquad\qquad+\frac{1}{2}p(p-1)(\delta f_{s}(X_s^*,u_s^*)+L|X_s^*|^{p})\bigg)\ds\bigg]\nn\\
&\qquad\;+M^{2}\E\bigg[\int_0^{\theta_{k}} e^{as}\Big( \delta f_{s}(X_s^*,u_s^*)+L|X_s^*|^{p}\Big)\ds\bigg]\nn\\
&\leq x^p+\E\bigg[\sup_{t\in[0,\theta_{k}]}\int_0^{t}e^{as}
\Big(a+\frac{1}{2}p^{2}L+\frac{1}{2}p+M^{2}L\Big)|X_s^{*}|^p\ds\bigg]\\
&\qquad\;+\Big(\frac{1}{2}p^{2}+M^{2}\Big)\delta\E\bigg[ \sup_{t\in[0,\theta_{k}]}\int_0^{t}e^{as} f_{s}(X_s^*,u_s^*)\ds\bigg]\nn\\
&\leq x^p+\Big(\frac{1}{2}p^{2}+M^{2}\Big)\delta e^{aT}\E\bigg[ \int_0^{T}f_{s}(X_s^*,u_s^*)\ds\bigg]\leq c,
\end{align*}
where the last inequality is due to Claim (ii) and non-negativity of $P_{1}$ and $P_{2}$ and $f_{s}$.
Now applying Fatou's lemma, the proof of Claim (iii) in Case I is complete.

\underline{Case II.} In this case, $\sigma(\pm1,\cdot)$ is bounded. Recall that $x>0$ and \eqref{optimalx}, applying H\"older's inequality yields
\begin{align*}
\E\Big[\sup_{t\in[0,T]}|X_t^*|^p\Big]
&=x^{p} \E\bigg[\sup_{t\in[0,T]}\exp\bigg\{\int_0^t\bigg(p b_{s}(1, \hat v_1)-\frac{1}{2}p|\sigma_{s}(1, \hat v_1)|^2\bigg)\ds+\int_0^tp\sigma_{s}(1, \hat v_1)^{\top}\dw_s \bigg\} \bigg]\\
&\leq c\E\bigg[\sup_{t\in[0,T]}\exp\bigg\{ \int_0^t\bigg(2p b_{s}(1, \hat v_1)+(2p^{2}-p)|\sigma_{s}(1, \hat v_1)|^2\bigg)\ds\bigg\}\bigg]^{1/2}\\
&\qquad\times\E\bigg[\sup_{t\in[0,T]}\exp\bigg\{-\int_0^t 2p^{2}|\sigma_{s}(1, \hat v_1)|^2\ds+\int_0^t2p\sigma_{s}(1, \hat v_1)^{\top}\dw_s \bigg\} \bigg]^{1/2}.
\end{align*}
The estimate \eqref{sigbup1case2} holds in Case II, so
\begin{align*}
&\quad\;\E\bigg[\sup_{t\in[0,T]}\exp\bigg\{ \int_0^t\bigg(2p b_{s}(1, \hat v_1)+(2p^{2}-p)|\sigma_{s}(1, \hat v_1)|^2\bigg)\ds\bigg\}\bigg]\\
&\leq \E\bigg[\exp\bigg\{ c\int_0^T(1+|\lm_1|)\ds\bigg\}\bigg]<\infty,
\end{align*}
where the last inequality is due to standard estimate for BMO martingales.
It is left to show
\begin{align}\label{supy}
\E\bigg[\sup_{t\in[0,T]}Y_{t}\bigg]<\infty,
\end{align}
where
\begin{align*}
Y_{t}=\exp\bigg\{-\int_0^t2p^{2}|\sigma_{s}(1, \hat v_1)|^2\ds+\int_0^t 2p\sigma_{s}(1, \hat v_1)^{\top}\dw_s \bigg\}.
\end{align*}
Since $\sigma(\pm1,\cdot)$ is bounded, $Y$ is a martingale with $\E[Y_T^{k}]<\infty$ for any $k\geq 1$. Then the estimate \eqref{supy} comes from martingale inequality immediately. This complete the proof of Claim (iii) in Case II.

\underline{Case III.}
Thanks to Assumptions \ref{homocondition} and \ref{solvablecondition3}, we have
\begin{align}\label{qua2used}
p|X_s^*|^{p-1}|b_{s}(X_s^*,u_s^*)|
&\leq \delta p|X_s^*|^{p-2}\idd{X_s^*\neq0}|\sigma_{s}(X_s^*,u_s^*)|^2+L|X_s^*|^{p}.
\end{align}
Similar to Case I, applying \ito to $e^{Lt}|X_t^{*}|^p$ gives
\begin{align}
e^{Lt}|X_t^{*}|^p& =x^p+\int_0^{t}e^{Ls}
\bigg(L|X_s^{*}|^p+p|X_s^*|^{p-1}\sgn{X_s^*}b_{s}(X_s^*,u_s^*)\nn\\
&\qquad\qquad\qquad\qquad+\frac{1}{2}p(p-1)|X_s^*|^{p-2}\idd{X_s^*\neq0}|\sigma_{s}(X_s^*,u_s^*)|^2\bigg)\ds\nn\\
&\qquad\;+\int_0^{t}e^{Ls}p|X_s^*|^{p-1}\sgn{X_s^*}\sigma_{s}(X_s^*,u_s^*)^{\top}\dw_s\nn\\
&\geq x^p+\Big(\frac{1}{2}(p-1)-\delta\Big)p\int_0^{t}e^{Ls} |X_s^*|^{p-2}\idd{X_s^*\neq0}|\sigma_{s}(X_s^*,u_s^*)|^2\ds\nn\\
&\qquad\;+\int_0^{t}e^{Ls}p|X_s^*|^{p-1}\sgn{X_s^*}\sigma_{s}(X_s^*,u_s^*)^{\top}\dw_s.\label{ito2}
\end{align}
Define
\begin{align*}
\theta_{k}:=\inf\bigg\{&t\geq0\;\bigg|\;|X_t^*|+\int_{0}^{t}|\sigma_{s}(1, \hat v_{1})|^{2}\ds
> k \bigg\}\wedge T.
\end{align*}
Since $X^*$ is continuous, by \eqref{sigup1}
we see $\theta_{k}$ is a sequence of stopping times that increasingly converges to $T$ as $k\to\infty$.
Noticing
\begin{align*}
|X_s^*|^{p-2}\idd{X_s^*\neq0}|\sigma_{s}(X_s^*,u_s^*)|^2
&=|X_s^*|^{p}|\sigma_{s}(1, \hat v_{1})|^2,
\end{align*}
recalling $\delta< \hf (p-1)$ and \eqref{uniboundstopping}, it then follows from \eqref{ito2} that
\begin{align*}
\E\bigg[\int_0^{\theta_{k}}e^{Ls}
|X_s^*|^{p-2}\idd{X_s^*\neq0}|\sigma_{s}(X_s^*,u_s^*)|^2\ds\bigg]\leq \frac{\E\Big[e^{L\theta_{k}}|X_{\theta_{k}}^{*}|^p\Big]}{(\frac{1}{2}(p-1)-\delta)p}\leq \frac{Ke^{LT}}{(\frac{1}{2}(p-1)-\delta)p}.
\end{align*}
%
We notice \eqref{BDG1} holds except for the last inequality, so it follows from \eqref{uniboundstopping}, \eqref{qua2used}, \eqref{ito2} and the above estimate that
\begin{align*}
&\quad\;\E\bigg[\sup_{t\in[0,\theta_{k}]}e^{Lt}|X_t^{*}|^p\bigg]\\
&\leq x^p+ \hf\E\bigg[\sup_{t\in[0,\theta_{k}]}e^{Lt}|X_t^{*}|^p\bigg]
+M^{2}\E\bigg[\int_0^{\theta_{k}} e^{Ls}\Big(|X_s^*|^{p-2}\idd{X^*_s\neq0}|\sigma_{s}(X_s^*,u_s^*)|^2+2L|X_s^*|^p\Big)\ds\bigg]\\
&\leq x^p+ \hf\E\bigg[\sup_{t\in[0,\theta_{k}]}e^{Lt}|X_t^{*}|^{p}\bigg]
+M^{2}\bigg(\frac{Ke^{LT}}{(\frac{1}{2}(p-1)-\delta)p}+2KTe^{LT}\bigg).
\end{align*}
In this estimate the constants are independent of $k$, so Claim (iii) in Case III follows from Fatou's lemma.
\end{proof}

\section{Concluding remarks}\label{conclude}
%
In this paper, we introduced and solved a new class of homogeneous stochastic control problems with cone control constraints under three distinct scenarios. We provided explicit expressions for the optimal value and the state feedback form of the optimal control in terms of solutions to two associated BSDEs, established through a rigorous verification procedure. The solvability of these BSDEs is of independent interest from the perspective of BSDE theory.

There are several promising directions for future research. First, our analysis is restricted to the one-dimensional state process in \eqref{state1}; it would be interesting to investigate   the problem \eqref{prob1} with a multi-dimensional state process. Second, when the state process exhibits discontinuities, such as those induced by Poisson jumps, new methods may be required to solve problem \eqref{prob1}. Exploring these extensions could yield further valuable insights.

%

\end{document}